\documentclass[11pt,a4paper]{article}
\usepackage[english]{babel}
\usepackage{amsmath,amsfonts,amsthm,amssymb,amsbsy,upref,color,graphicx,amscd,hyperref,makeidx,enumerate,bbm,comment}
\usepackage[a4paper,margin=2.0truecm]{geometry}
\usepackage[active]{srcltx}
\usepackage[latin1]{inputenc}

\def\h{{\mathcal H^{d_2-1}}}
\def\C{{\mathcal C}}
\def\A{\mathcal{A}}
\def\E{\mathcal{E}}

\def\HH{\mathcal{H}}
\def\K{\mathcal{K}}

\def\ds{\displaystyle}
\def\eps{{\varepsilon}}
\def\O{\Omega}
\def\N{\mathbb{N}}
\def\R{\mathbb{R}}

\newcommand{\be}{\begin{equation}}
\newcommand{\ee}{\end{equation}}
\newcommand{\bib}[4]{\bibitem{#1}{\sc#2: }{\it#3. }{#4.}}

\numberwithin{equation}{section} \theoremstyle{plain}
\newtheorem{teo}{Theorem}[section]
\newtheorem{lemma}[teo]{Lemma}

\newtheorem{prop}[teo]{Proposition}

\newtheorem{conj}[teo]{Conjecture}

\theoremstyle{remark}
\newtheorem{oss}[teo]{Remark}

\theoremstyle{figure}

\newenvironment{ack}{{\bf Acknowledgements.}}

\title{Optimization problems involving the first Dirichlet eigenvalue and the torsional rigidity}

\author{Michiel van den Berg, Giuseppe Buttazzo,  Bozhidar Velichkov}

\date{28 October 2014}

\begin{document}
\maketitle

\begin{abstract}
We present some open problems and obtain some partial results for spectral optimization problems involving measure, torsional rigidity and first Dirichlet eigenvalue.
\end{abstract}



\section{Introduction}\label{sintro}

A shape optimization problem can be written in the very general form
$$\min\big\{F(\O)\ :\ \O\in\A\big\},$$
where $\A$ is a class of admissible domains and $F$ is a cost functional defined on $\A$. We consider in the present paper the case where the cost functional $F$ is related to the solution of an elliptic equation and involves the spectrum of the related elliptic operator. We speak in this case of {\it spectral optimization problems}. Shape optimization problems of spectral type have been widely considered in the literature; we mention for instance the papers \cite{bulbk}, \cite{bubuveihp}, \cite{bbf99}, \cite{budm91}, \cite{budm93}, \cite{buvejfa}, \cite{buve12}, \cite{pm10}, and we refer to the books \cite{bubu05}, \cite{hen06}, \cite{hepi05}, and to the survey papers \cite{ash99}, \cite{but10}, \cite{hen03}, where the reader can find a complete list of references and details.

In the present paper we restrict ourselves for simplicity to the Laplace operator $-\Delta$ with Dirichlet boundary conditions. Furthermore we shall assume that the admissible domains $\O$ are a priori contained in a given {\it bounded} domain $D\subset\R^d$. This assumption greatly simplifies several existence results that otherwise would require additional considerations in terms of concentration-compactness arguments \cite{bulbk}, \cite{tesi}.

The most natural constraint to consider on the class of admissible domains is an inequality on their Lebesgue measure. Our admissible class $\A$ is then
$$\A=\big\{\O\subset D\ :\ |\O|\le1\big\}.$$
Other kinds of constraints are also possible, but we concentrate here to the one above, referring the reader interested in possible variants to the books and papers quoted above.

The following two classes of cost functionals are the main ones considered in the literature.

\medskip{\it Integral functionals.} Given a right-hand side $f\in L^2(D)$, for every $\O\in\A$ let $u_\O$ be the unique solution of the elliptic PDE
$$-\Delta u=f\hbox{ in }\O,\qquad u\in H^1_0(\O).$$
The integral cost functionals are of the form
$$F(\O)=\int_\O j\big(x,u_\O(x),\nabla u_\O(x)\big)\,dx,$$
where $j$ is a suitable integrand that we assume convex in the gradient variable. We also assume that $j$ is bounded from below by
$$j(x,s,z)\ge-a(x)-c|s|^2,$$
with $a\in L^1(D)$ and $c$ smaller than the first Dirichlet eigenvalue of the Laplace operator $-\Delta$ in $D$. For instance, the energy $\E_f(\O)$ defined by
$$\E_f(\O)=\inf\left\{\int_D\Big(\frac12|\nabla u|^2-f(x)u\Big)\,dx\ :\ u\in H^1_0(\O)\right\},$$
belongs to this class since, integrating by parts its Euler-Lagrange equation, we have that
$$\E_f(\O)=-\frac12\int_D f(x)u_\O\,dx,$$
which corresponds to the integral functional above with
$$j(x,s,z)=-\frac12 f(x)s.$$
The case $f=1$ is particularly interesting for our purposes. We denote by $w_\O$ the {\it torsion function}, that is the solution of the PDE
$$-\Delta u=1\hbox{ in }\O,\qquad u\in H^1_0(\O),$$
and by the {\it torsional rigidity} $T(\O)$ the $L_1$ norm of $w_{\O}$,
$$T(\O)=\int_\O w_\O\,dx=-2\E_1(\O).$$

\medskip{\it Spectral functionals.} For every admissible domain $\O\in\A$ we consider the spectrum $\Lambda(\O)$ of the Laplace operator $-\Delta$ on $H^1_0(\O)$. Since $\O$ has a finite measure, the operator $-\Delta$ has a compact resolvent and so its spectrum $\Lambda(\O)$ is discrete:
$$\Lambda(\O)=\big(\lambda_1(\O),\lambda_2(\O),\dots\big),$$
where $\lambda_k(\O)$ are the eigenvalues counted with their multiplicity. The spectral cost functionals we may consider are of the form
$$F(\O)=\Phi\big(\Lambda(\O)\big),$$
for a suitable function $\Phi:\R^\N\to\overline{\R}$. For instance, taking $\Phi(\Lambda)=\lambda_k(\O)$ we obtain
$$F(\O)=\lambda_k(\O).$$

\medskip We take the torsional rigidity $T(\O)$ and the first eigenvalue $\lambda_1(\O)$ as prototypes of the two classes above and we concentrate our attention on cost functionals that depend on both of them. We note that, by the maximum principle, when $\O$ increases $T(\O)$ increases, while $\lambda_1(\O)$ decreases.

\section{Statement of the problem}\label{sposit}

The optimization problems we want to consider are of the form
\be\label{optpb} \min\left\{\Phi\big(\lambda_1(\O),T(\O)\big)\ :\
\O\subset D,\ |\O|\le1\right\},
\ee
where we have normalized the constraint on the Lebesgue measure of $\O$, and where $\Phi$ is a given continuous (or lower semi-continuous) and non-negative function. In the rest of this paper we often take for simplicity $D=\R^d$, even if most of the results are valid in the general case. For instance, taking $\Phi(a,b)=ka+b$ with $k$ a fixed positive constant, the quantity we aim to minimize becomes
$$k\lambda_1(\O)+T(\O)\qquad\hbox{with $\O\subset D,$ and }|\O|\le1.$$

\begin{oss}\label{exbdm}
If the function $\Phi(a,b)$ is increasing with respect to $a$ and decreasing with respect to $b$, then the cost functional
$$F(\O)=\Phi\big(\lambda_1(\O),T(\O)\big)$$
turns out to be decreasing with respect to the set inclusion. Since both the torsional rigidity and the first eigenvalue are $\gamma$-continuous functionals and the function $\Phi$ is assumed
lower semi-continuous, we can apply the existence result of \cite{budm93}, which provides the existence of an optimal domain.
\end{oss}

In general, if the function $\Phi$ does not verify the monotonicity property of Remark \ref{exbdm}, then the existence of an optimal domain is an open problem, and the aim of this paper is to discuss this issue. For simplicity of the presentation we limit ourselves to the two-dimensional case $d=2$. The case of general $d$ does not present particular difficulties but requires the use of several $d-$ dependent exponents.

\begin{oss}\label{facts}
The following facts are well known.

\begin{itemize}
\item[i)] If $B$ is a disk in $\R^2$ we have
$$T(B)=\frac{1}{8\pi}|B|^2.$$

\item[ii)] If $j_{0,1}\approx2.405$ is the first positive zero of the Bessel functions $J_0(x)$ and $B$ is a disk of $\R^2$ we have
$$\lambda_1(B)=\frac{\pi}{|B|}j^2_{0,1}.$$

\item[iii)] The torsional rigidity $T(\O)$ scales as
$$T(t\O)=t^4T(\O),\qquad\forall t>0.$$

\item[iv)] The first eigenvalue $\lambda_1(\O)$ scales as
$$\lambda_1(t\O)=t^{-2}\lambda_1(\O),\qquad\forall t>0.$$

\item[v)] For every domain $\O$ of $\R^2$ and any disk $B$ we have
$$|\O|^{-2}T(\O)\le|B|^{-2}T(B)=\frac{1}{8\pi}.$$

\item[vi)] For every domain $\O$ of $\R^2$ and any disk $B$ we have (Faber-Krahn inequality)
$$|\O|\lambda_1(\O)\ge|B|\lambda_1(B)=\pi j^2_{0,1}.$$

\item[vii)] A more delicate inequality is the so-called Kohler-Jobin inequality (see \cite{kojo}, \cite{brasco}): for any domain $\O$ of $\R^2$ and any disk $B$ we have
$$\lambda^2_1(\O)T(\O)\ge\lambda^2_1(B)T(B)=\frac{\pi}{8}j^4_{0,1}.$$
This had been previously conjectured by G. P\'{o}lya and G.Szeg\"{o} \cite{PS}.
\end{itemize}
\end{oss}

We recall the following inequality, well known for planar regions (Section 5.4 in \cite{PS}), between torsional rigidity and first eigenvalue.

\begin{prop}\label{ineq1}
For every domain $\O\subset\R^d$ we have
$$\lambda_1(\O)T(\O)\le|\O|.$$
\end{prop}

\begin{proof}
By definition, $\lambda_1(\O)$ is the infimum of the Rayleigh quotient
$$\int_\O|\nabla u|^2\,dx\bigg/\int_\O u^2\,dx\qquad\hbox{over all }u\in H^1_0(\O),\ u\ne0.$$
Taking as $u$ the torsion function $w_\O$, we have
$$\lambda_1(\O)\le\int_\O|\nabla w_\O|^2\,dx\bigg/\int_\O w_\O^2\,dx.$$
Since $-\Delta w_\O=1$, an integration by parts gives
$$\int_\O|\nabla w_\O|^2\,dx=\int_\O w_\O\,dx=T(\O),$$
while the H\"older inequality gives
$$\int_\O w^2_\O\,dx\ge\frac{1}{|\O|}\left(\int_\O w_\O\,dx\right)^2=\frac{1}{|\O|}\big(T(\O)\big)^2.$$
Summarizing, we have
$$\lambda_1(\O)\le\frac{|\O|}{T(\O)}$$
as required.
\end{proof}

\begin{oss}
The infimum of $\lambda_1(\O)T(\O)$ over open sets $\O$ of prescribed measure is zero. To see this, let $\O_n$ be the disjoint union of one ball of volume $1/n$ and $n(n-1)$ balls of volume $1/n^2$. Then the radius $R_n$ of the ball of volume $1/n$ is $(n\omega_d)^{-1/d}$ while the radius $r_n$ of the balls of volume $1/n^2$ is $(n^2\omega_d)^{-1/d}$, so that $|\O_n|=1$,
$$\lambda_1(\O_n)=\lambda_1(B_{R_n})=\frac{1}{R_n^2}\lambda_1(B_1)=(n\omega_d)^{2/d}\lambda_1(B_1),$$
and
\[\begin{split}
T(\O_n)&=T(B_{R_n})+n(n-1)T(B_{r_n})=T(B_1)\big(R_n^{d+2}+n(n-1)r_n^{d+2}\big)\\
&=T(B_1)\omega_d^{-1-2/d}\big(n^{-1-2/d}+(n-1)n^{-1-4/d}\big).
\end{split}\]
Therefore
$$\lambda_1(\O_n)T(\O_n)=\frac{\lambda_1(B_1)T(B_1)}{\omega_d}\frac{n^{2/d}+n-1}{n^{1+2/d}},$$
which vanishes as $n\to\infty$.
\end{oss}

In the next section we investigate the inequality of Proposition \ref{ineq1}.

\section{A sharp inequality between torsion and first eigenvalue}\label{sineq}

We define the constant
$$\K_d=\sup\left\{\frac{\lambda_1(\O)T(\O)}{|\O|}\ :\ \O\hbox{ open in }\R^d,\ |\O|<\infty\right\}.$$
We have seen in Proposition \ref{ineq1} that $\K_d\le1$. The question is if the constant $1$ can be improved.

Consider a ball $B$; performing the shape derivative as in \cite{hepi05}, keeping the volume of the perturbed shapes constant, we obtain that for every field $V(x)$
$$\partial[\lambda_1(B)T(B)](V)=T(B)\partial[\lambda_1(B)](V)+\lambda_1(B)\partial[T(B)](V)=C_B\int_{\partial B}V\cdot n\,d\HH^{d-1}$$
for a suitable constant $C_B$. Since the volume of the perturbed shapes is constant, we have
$$\int_{\partial B}V\cdot n\,d\HH^{d-1}=0,$$
where $\HH^{d-1}$ denotes $(d-1)$-dimensional Hausdorff measure. This shows that balls are stationary for the functional
$$F(\O)=\frac{\lambda_1(\O)T(\O)}{|\O|}.$$
Below we will show, by considering rectangles, that balls are not optimal. To do so we shall obtain a lower bound for the torsional rigidity of a rectangle.

\begin{prop}\label{rectangle}
In a rectangle $R_{a,b}=(-b/2,b/2)\times(-a/2,a/2)$ with $a\le b$ we have
$$T(R_{a,b})\ge\frac{a^3b}{12}-\frac{11a^4}{180}.$$
\end{prop}

\begin{proof}
Let us estimate the energy
$$\E_1(R_{a,b})=\inf\left\{\int_{R_{a,b}}\left(\frac12|\nabla u|^2-u\right)\,dx\,dy\ :\ u\in H^1_0(R_{a,b})\right\}$$
by taking the function
$$u(x,y)=\frac{a^2-4y^2}{8}\theta(x),$$
where $\theta(x)$ is defined by
$$\theta(x)=\begin{cases}
1&\hbox{,if }|x|\le(b-a)/2\\
(b-2|x|)/a&\hbox{,otherwise.}
\end{cases}$$
We have
$$|\nabla u|^2=\left(\frac{a^2-4y^2}{8}\right)^2|\theta'(x)|^2+y^2|\theta(x)|^2,$$
so that
\[\begin{split}
\E_1(R_{a,b})&\le2\int_0^{a/2}\left(\frac{a^2-4y^2}{8}\right)^2\,dy\int_0^{b/2}|\theta'(x)|^2\,dx+2\int_0^{a/2}y^2\,dy\int_0^{b/2}|\theta(x)|^2\,dx\\
&\qquad-4\int_0^{a/2}\frac{a^2-4y^2}{8}\,dy\int_0^{b/2}\theta(x)\,dx\\
&=\frac{a^4}{60}+\frac{a^3}{12}\left(\frac{b-a}{2}+\frac{a}{6}\right)-\frac{a^3}{6}\left(\frac{b-a}{2}+\frac{a}{4}\right)\\
&=-\frac{a^3b}{24}+\frac{11a^4}{360}.
\end{split}\]
The desired inequality follows since $T(R_{a,b})=-2\E_1(R_{a,b})$.
\end{proof}

In $d$-dimensions we have the following.

\begin{prop}\label{prop}
If $\O_\eps=\omega\times(-\eps/2,\eps,2)$, where $\omega$ is a convex set in $\R^{d-1}$ with $|\omega|<\infty$, then
$$T(\O_\eps)=\frac{\eps^3}{12}|\omega|+O(\eps^4),\qquad\epsilon\downarrow 0.$$
\end{prop}
We defer the proof to Section \ref{app}.

For a ball of radius $R$ we have
\be\label{dd}
\lambda_1(B)=\frac{j^2_{d/2-1,1}}{R^2},\qquad T(B)=\frac{\omega_d R^{d+2}}{d(d+2)},\qquad|B|=\omega_dR^d,
\ee
so that
$$F(B)=\frac{\lambda_1(B)T(B)}{|B|}=\frac{j^2_{d/2-1,1}}{d(d+2)}:=\alpha_d$$
For instance, we have
$$\alpha_2\approx0.723,\qquad\alpha_3\approx0.658,\qquad\alpha_4\approx0.612.$$
Moreover, since $j_{\nu,1}=\nu+O(\nu^{1/3}),\ \nu\to\infty$, we have that $\lim_{d\to\infty}\alpha_d=\frac14$. A plot of $\alpha_d$ is given in Figure \ref{fig1}.
\begin{figure}[h!]
\centering {\includegraphics[width=9cm]{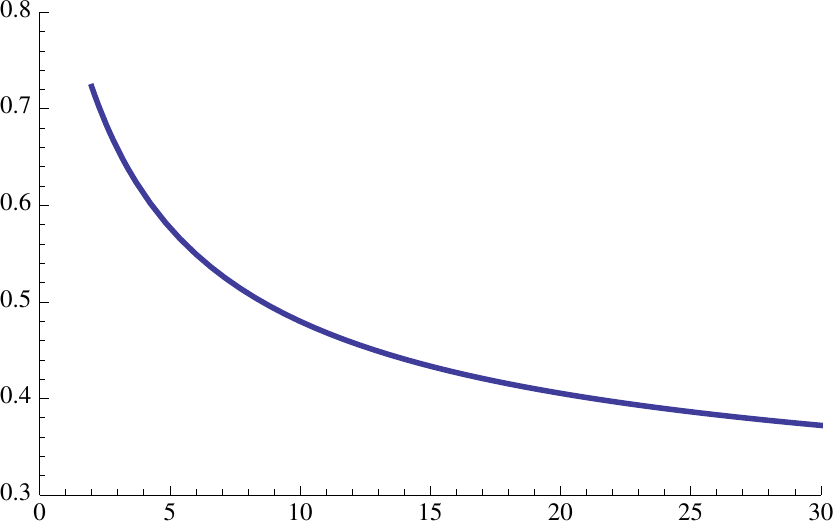}}
\caption{The plot
of $\alpha_d$ for $2\le d\le30$.}\label{fig1}
\end{figure}

We now consider a slab $\O_\eps=\omega\times(0,\eps)$ of thickness $\eps\to0$. We have by separation of variables and Proposition \ref{prop} that
$$\lambda_1(\O_\eps)=\frac{\pi^2}{\eps^2}+\lambda_1(\omega)\approx \frac{\pi^2}{\eps^2},\qquad T(\O_\eps)\approx\frac{\eps^3|\omega|}{12},\qquad|\O_\eps|=\eps|\omega|,$$
so that
$$F(\O_\eps)\approx\frac{\pi^2}{12}\approx0.822.$$
This shows that in any dimension the slab is better than the ball. Using domains in $\R^d$ with $k$ small dimensions and $d-k$ large dimensions does not improve the value of the cost functional $F$. In fact, if $\omega$ is a convex domain in $\R^{d-k}$ and $B_k(\eps)$ a ball in $\R^k$, then by Theorem \ref{the2} with $\O_\eps=\omega\times B_k(\eps)$ we have that
$$\lambda_1(\O_\eps)\approx\frac{1}{\eps^2}\lambda_1\big(B_k(1)\big),\qquad T(\O_\eps)\approx\eps^{k+2}|\omega|T(B_k(1)),\qquad|\O_\eps|=\eps^k|\omega||B_k(1)|,$$
so that
$$F(\O_\eps)\approx\frac{j^2_{k/2-1,1}}{k(k+2)}\le\frac{\pi^2}{12}.$$
This supports the following.

\begin{conj}\label{conject}
For every dimension $d$ we have $\K_d=\pi^2/12,$ and no domain in $\R^d$ maximizes the functional $F$ for $d>1$. The maximal value $\K_d$ is asymptotically reached by a thin slab
$\O_\eps=\omega\times(0,\eps)$, with $\omega\subset\R^{d-1}$, as $\eps\to0$.
\end{conj}

\section{The attainable set}\label{splotting}

In this section we bound the measure by $|\O|\le1$. Our goal is to plot the subset of $\R^2$ whose coordinates are the eigenvalue $\lambda_1(\O)$ and the torsion $T(\O)$. It is convenient to change coordinates and to set for a given admissible domain $\O,$
$$x=\lambda_1(\O),\qquad y=\big(\lambda_1(\O)T(\O)\big)^{-1}.$$
In addition, define
$$E=\left\{(x,y)\in\R^2\ :\ x=\lambda_1(\O),\ y=\big(\lambda_1(\O)T(\O)\big)^{-1}\hbox{ for some $\O$ with }|\O|\le1\right\}.$$
Therefore, the optimization problem \eqref{optpb} can be rewritten as
$$\min\left\{\Phi\big(x,1/(xy)\big)\ :\ (x,y)\in E\right\}.$$

\begin{conj}\label{conject1}
The set $E$ is closed.
\end{conj}

We remark that the conjecture above, if true, would imply the existence of a solution of the optimization problem \eqref{optpb} for many functions $\Phi$. Below we will analyze the variational problem in case  $\Phi(x,y)=kx+\frac{1}{xy},$ where $k>0$.

\begin{teo}\label{closedex}
Let $d=2,3,\cdots$, and let
$$k^*_d=\frac{1}{2d\omega_d^{4/d}j^2_{d/2-1,1}}.$$
Consider the optimization problem
\be\label{pbk}
\min\left\{k\lambda_1(\O)+T(\O)\ :\ |\O|\le1\right\}.
\ee
If $0<k\le k^*_d$ then the ball with radius
\be\label{e2}
R_k=\left(\frac{2kdj^2_{d/2-1,1}}{\omega_d}\right)^{1/(d+4)}
\ee
is the unique minimizer (modulo translations and sets of capacity $0$).

\noindent If $k> k^*_d$ then the ball $B$ with measure $1$ is the unique minimizer.
\end{teo}

\begin{proof}
Consider the problem \eqref{pbk} without the measure constraint \be\label{without}
\min\left\{k\lambda_1(\O)+T(\O)\ :\ \O\subset\R^d\right\}. \ee
Taking $t\O$ instead of $\O$ gives that
$$k\lambda_1(t\O)+T(t\O)=kt^{-2}\lambda_1(\O)+t^{d+2}T(\O).$$
The optimal $t$ which minimizes this expression is given by
$$t=\left(\frac{2k\lambda_1(\O)}{(d+2)T(\O)}\right)^{1/(d+4)}.$$
Hence \eqref{without} equals
\be\label{aftert}
\min\left\{(d+4)\left(\frac{k^{d+2}}{4(d+2)^{d+2}}T^2(\O)\lambda_1^{d+2}(\O)\right)^{1/(d+4)}:\
\O\subset\R^d\ \right\}.
\ee
By the Kohler-Jobin inequality in $\R^d$, the minimum in \eqref{aftert} is attained by any ball. Therefore the minimum in \eqref{without} is given by a ball $B_R$ such that
$$\left(\frac{2k\lambda_1(B_R)}{(d+2)T(B_R)}\right)^{1/(d+4)}=1.$$
This gives \eqref{e2}. We conclude that the measure constrained problem \eqref{pbk} admits the ball $B_{R_k}$ as a solution whenever $\omega_dR_k^d\le1$. That is $k\le k^*_d.$

Next consider the case $k>k^*_d$. Let $B$ be the open ball with measure $1$. It is clear that
$$\min\{k\lambda_1(\O)+T(\O)\ :\ |\O|\le 1\}
\le k\lambda_1(B)+T(B).$$
To prove the converse we note that for $k>k^*_d$,
\begin{align}\label{e4}
\min\{k\lambda_1&(\O)+T(\O):|\O|\le 1\}\nonumber\\
&\ge\min\{(k-k^*_d)\lambda_1(\O):|\O|\le1\}
+\min\{k^*_d\lambda_1(\O)+T(\O):|\O|\le 1\}.
\end{align}
The minimum in the first term in the right hand side of \eqref{e4} is attained for $B$ by Faber-Krahn, whereas the minimum in second term is attained for $B_{R_{k^*_d}}$ by our previous unconstrained calculation. Since $|B_{R_{k^*_d}}|=|B|=1$ we have by \eqref{e4} that
\begin{align*}
\min\{k\lambda_1&(\O)+T(\O):|\O|\le 1\}\nonumber\\
&\ge(k-k^*_d)\lambda_1(B)+ k^*_d\lambda_1(B)+T(B)\nonumber\\
&=k\lambda_1(B)+T(B).
\end{align*}
Uniqueness of the above minimizers follows by uniqueness of Faber-Krahn and Kohler-Jobin.
\end{proof}

It is interesting to replace the first eigenvalue in \eqref{pbk} be a higher eigenvalue. We have the following for the second eigenvalue.

\begin{teo}\label{lambda2}
Let $d=2,3,\cdots$, and let
$$l^*_d=\frac{1}{2d(2\omega_d)^{4/d}j^2_{d/2-1,1}}.$$
Consider the optimization problem
\be\label{pbk2}
\min\left\{l\lambda_2(\O)+T(\O)\ :\ |\O|\le1\right\}.
\ee
If $0<l\le l^*_d$ then the union of two disjoint balls with radii
\be\label{e7}
R_l=\left(\frac{ldj^2_{d/2-1,1}}{\omega_d}\right)^{1/(d+4)}
\ee
is the unique minimizer (modulo translations and sets of capacity $0$).

\noindent If $l> l^*_d$ then union of two disjoint balls with measure $1/2$ each is the unique minimizer.
\end{teo}

\begin{proof}
First consider the unconstrained problem
\be\label{e8}
\min\left\{l\lambda_1(\O)+T(\O)\ :\ \O\subset\R^d\right\}. \ee
Taking $t\O$ instead of $\O$ gives that
$$l\lambda_2(t\O)+T(t\O)=lt^{-2}\lambda_2(\O)+t^{d+2}T(\O).$$
The optimal $t$ which minimizes this expression is given by
$$t=\left(\frac{2l\lambda_2(\O)}{(d+2)T(\O)}\right)^{1/(d+4)}.$$
Hence \eqref{e8} equals
\be\label{e10}
\min\left\{(d+4)\left(\frac{l^{d+2}}{4(d+2)^{d+2}}T^2(\O)\lambda_2^{d+2}(\O)\right)^{1/(d+4)}:\
\O\subset\R^d\ \right\}.
\ee
It follows by the Kohler-Jobin inequality, see for example Lemma 6 in \cite{vdB7}, that the minimizer of \eqref{e10} is attained by the union of two disjoint balls $B_R$ and $B'_R$ with the same radius. Since $\lambda_2(B_R\cup B'_R)=\lambda_1(B_R)$ and $T(B_R\cup B'_R)=2T(B_R)$ we have, using \eqref{dd}, that the radii of these balls are given by \eqref{e7}. We conclude that the measure constrained problem \eqref{pbk2} admits the union of two disjoint balls with equal radius $R_l$ as a solution whenever $2\omega_dR_l^d\le1$. That is $l\le l^*_d.$

Next consider the case $l>l^*_d$. Let $\O$ be the union of two disjoint balls $B$ and $B'$ with measure $1/2$ each. Then
$$\min\{l\lambda_2(\O)+T(\O)\ :\ |\O|\le 1\}\le l\lambda_1(B)+2T(B).$$
To prove the converse we note that for $l>l^*_d$,
\begin{align}\label{e12}
\min\{l\lambda_2&(\O)+T(\O)\ :\ |\O|\le 1\}\nonumber\\
&\ge\min\{(l-l^*_d)\lambda_2(\O)\ :\ |\O|\le1\}
+\min\{l^*_d\lambda_2(\O)+T(\O)\ :\ |\O|\le1\}.
\end{align}
The minimum in the first term in the right hand side of \eqref{e12} is attained for $B\cup B'$ by the Krahn-Szeg\"o inequality, whereas the minimum in second term is attained for the union of two disjoint balls with radius $R_{l^*_d}$ by our previous unconstrained calculation. Since $|B_{R_{l^*_d}}|=1/2=|B|=|B'|$ we have by \eqref{e12} that
\begin{align*}
\min\{l\lambda_2(\O)+T(\O)\ :\ |\O|\le 1\}&\ge(l-l^*_d)\lambda_1(B)+ l^*_d\lambda_1(B)+2T(B)\nonumber\\
&=l\lambda_1(B)+2T(B).
\end{align*}
Uniqueness of the above minimizers follows by uniqueness of Krahn-Szeg\"o and Kohler-Jobin for the second eigenvalue.
\end{proof}

To replace the first eigenvalue in \eqref{pbk} be the $j$'th eigenvalue ($j>2$) is a very difficult problem since we do not know the minimizers of the $j$'th Dirichlet eigenvalue with a measure constraint nor the minimizer of the $j$'th Dirichlet eigenvalue a torsional rigidity constraint. However, if these two problems have a common minimizer then information similar to the above can be obtained.

Putting together the facts listed in Remark \ref{facts} we obtain the following inequalities.

\begin{itemize}

\item[(i)] By Faber-Krahn inequality we have $x\ge\pi j^2_{0,1}\approx18.168$.

\item[(ii)] By Conjecture \ref{conject} (if true) we have $y\ge12/\pi^2\approx1.216$.

\item[(iii)] By the bound on the torsion of Remark \ref{facts} v) we have $xy\ge8\pi\approx25.133$.

\item[(iv)] By the Kohler-Jobin inequality we have $y/x\le8/(\pi j^4_{0,1})\approx0.076$.

\item[(v)] The set $E$ is {\it conical}, that is if a point $(x_0,y_0)$ belongs to $E$, then all the half-line $\big\{(tx_0,ty_0)\ :\ t\ge1\big\}$ in contained in $E$. This follows by taking $\O_t=\O/t$ and by the scaling properties iii) and iv) of Remark \ref{facts}.

\item[(vi)] The set $E$ is {\it vertically convex}, that is if a point $(x_0,y_0)$ belongs to $E$, then all points $(x_0,ty_0)$ with $1\le t\le8/(\pi j^4_{0,1})$ belong to $E$. To see this fact, let $\O$ be a domain corresponding to the point $(x_0,y_0)\in E$. The {\it continuous Steiner symmetrization} path $\O_t$ (with $t\in[0,1]$) then continuously deforms the domain $\O=\O_0$ into a ball $B=\O_1$, preserving the Lebesgue measure and decreasing $\lambda_1(\O_t)$ (see \cite{brock} where this tool has been developed, and Section 6.3 of \cite{bubu05} for a short survey). The curve
$$x(t)=\lambda_1(\O_t),\qquad y(t)=\big(\lambda_1(\O_t)T(\O_t)\big)^{-1}$$
then connects the point $(x_0,y_0)$ to the Kohler-Jobin line $\left\{y=8x/(\pi j^4_{0,1})\right\}$, having $x(t)$ decreasing. Since $\big(x(t),y(t)\big)\in E$, the conicity of $E$ then implies vertical convexity.
\end{itemize}

\noindent A plot of the constraints above is presented in Figure \ref{fig2}.
\begin{figure}[h!]
\centering
{\includegraphics[width=10cm]{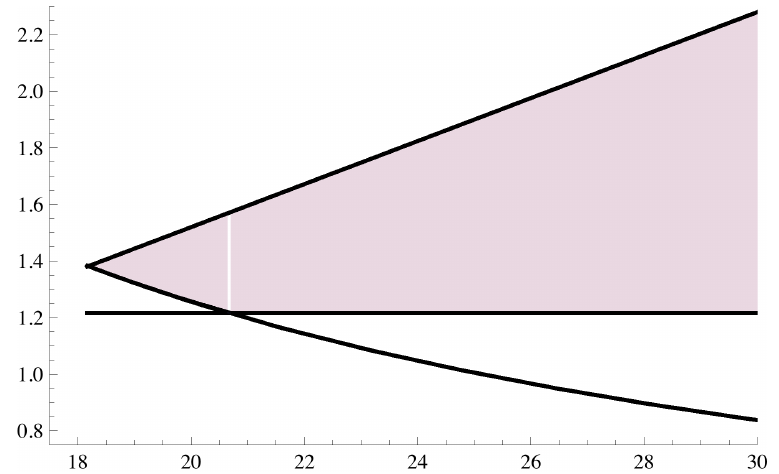}}
\caption{The admissible region $E$ is contained in the dark area.}\label{fig2}
\end{figure}
\noindent Some particular cases can be computed explicitly. Consider $d=2$, and let
$$\O=B_R\cup B_r, \hbox{with $B_R\cap B_r=\emptyset$, $r\le R$, and }\pi(R^2+r^2)=1.$$
An easy computation gives that
$$\lambda_1(\O)=\frac{j^2_{0,1}}{R^2},\qquad T(\O)=\frac{2\pi^2R^4-2\pi R^2+1}{8\pi},$$
so that the curve
$$y=\frac{8\pi x}{x^2-2\pi j^2_{0,1}x+2\pi^2j^4_{0,1}},\qquad\pi j^2_{0,1}\le x\le2\pi j^2_{0,1}$$
is contained in $E$ (see Figure \ref{fig3}).
\begin{figure}[h!]
\centering
{\includegraphics[width=10cm]{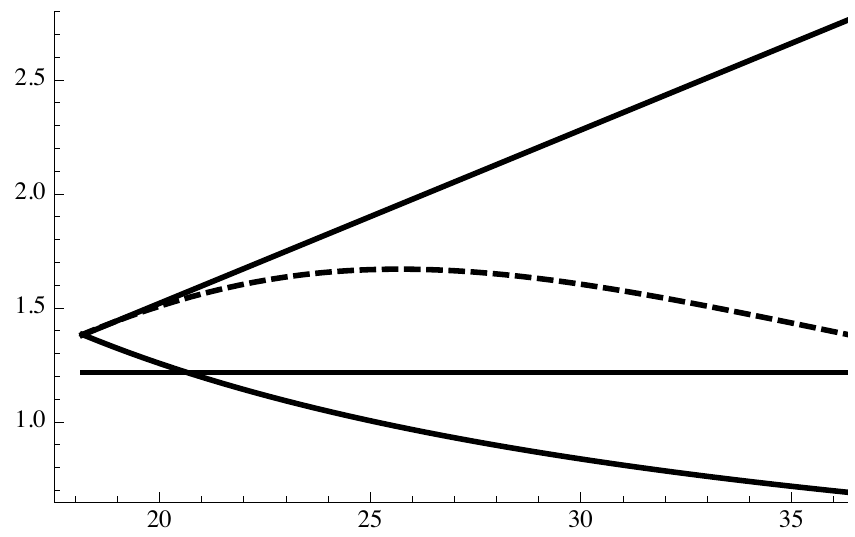}}
\caption{The dashed line corresponds to two disks of variable radii.}\label{fig3}
\end{figure}

\noindent If we consider the rectangle
$$\O=(0,b)\times(0,a)\hbox{,\ with $a\le b$, and }ab=1,$$
we have by Proposition \ref{rectangle}
$$\lambda_1(\O)=\pi^2\left(\frac{1}{a^2}+\frac{1}{b^2}\right)=\pi^2\left(\frac{1}{a^2}+a^2\right),\qquad T(\O)\ge\frac{a^3b}{12}-\frac{11a^4}{180}=\frac{a^2}{12}-\frac{11a^4}{180}.$$
Therefore
$$y\le h\big(x/(2\pi^2)\big),\qquad\hbox{where }h(t)=\frac{90}{\pi^2 t\left(11+15t-22t^2-(15+2t)\sqrt{t^2-1}\right)},\quad t\ge1.$$
By $E$ being conical the curve
$$y=h\big(x/(2\pi^2)\big)\qquad\pi^2\le x<+\infty$$
is contained in $E$ (see Figure \ref{fig4}).
\begin{figure}[h!]
\centering
{\includegraphics[width=10cm]{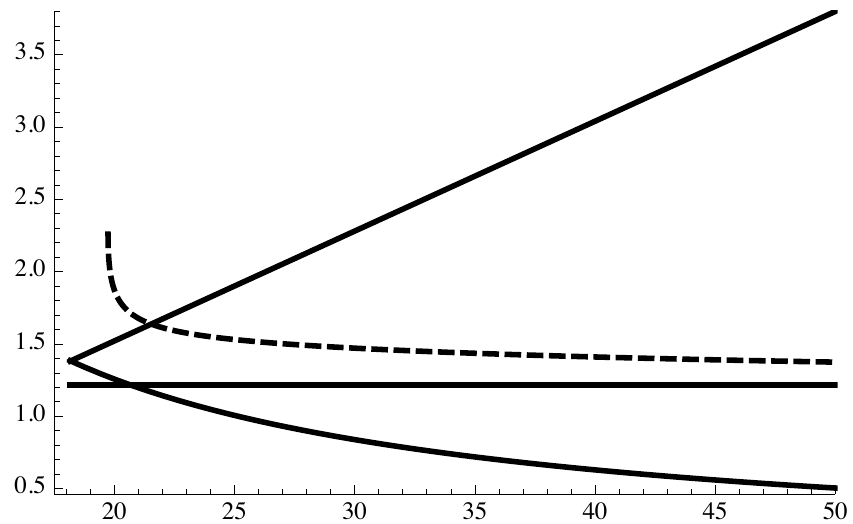}}
\caption{The dashed line is an upper bound to the line corresponding to rectangles.}\label{fig4}
\end{figure}

Besides the existence of optimal domains for the problem \eqref{optpb}, the regularity of optimal shapes is another very delicate and important issue. Very little is known about the regularity of optimal domains for spectral optimization problems (see for instance \cite{brla}, \cite{busub}, \cite{deve}, \cite{tesi}); the cases where only the first eigenvalue $\lambda_1(\O)$ and the torsion $T(\O)$ are involved could be simpler and perhaps allow to use the free boundary methods developed in \cite{altcaf}.

\section{Torsional rigidity and the heat equation}\label{app}

It is well known that the rich interplay between elliptic and parabolic partial differential equations provide tools for obtaining results in one field using tools from the other. See for
example the monograph by E. B. Davies \cite{davies}, and \cite{vdB4,vdB5,vdB6,vdB1,vdB2,vdB3} for some more recent results. In this section we use some heat equation tools to obtain new
estimates for the torsional rigidity. Before we do so we recall some basic facts relating the torsional rigidity to the heat equation. For an open set $\O$ in $\R^d$ with boundary $\partial \O$ we denote the Dirichlet heat kernel by $p_{\O}(x,y;t),\ x\in\O,\ y\in\O,\ t>0$. So
$$u_\O(x;t):=\int_\O p_\O(x,y;t)\,dy,$$
is the unique weak solution of
$$\begin{cases}
\ds\frac{\partial u}{\partial t}=\Delta u&x\in\O,\ t>0,\\
\lim_{t\downarrow0}u(x;t)=1&\hbox{in }L^2(\O),\\
u(x;t)=0&x\in\partial\O,\ t>0.
\end{cases}$$
The latter boundary condition holds at all regular points of $\partial\O$. We denote the heat content of $\O$ at time $t$ by
$$Q_{\O}(t)=\int_{\O}u_{\O}(x;t)\,dx.$$
Physically the heat content represents the amount of heat in $\O$ at time $t$ if $\O$ has initial temperature $1$, while $\partial\O$ is kept at temperature $0$ for all $t>0$. Since the Dirichlet heat kernel is non-negative, and monotone in $\O$ we have that
\be\label{a6}
0\le p_\O(x,y;t)\le p_{\R^d}(x,y;t)
=(4\pi t)^{-d/2}e^{-|x-y|^2/(4t)}.
\ee
It follows by either \eqref{a6} or by the maximum principle that
$$0\le u_\O(x;t)\le 1,$$
and that if $|\O|<\infty$ then
\be\label{a8}
0\le Q_\O(t)\le |\O|.
\ee
In the latter situation we also have an eigenfunction expansion for the Dirichlet heat kernel in terms of the Dirichlet eigenvalues $\lambda_1(\O)\le\lambda_2(\O)\le\cdots$, and a corresponding orthonormal set of eigenfunctions $\{\varphi_1,\varphi_2,\cdots\}$,
$$p_\O(x,y;t)=\sum_{j=1}^{\infty}e^{-t\lambda_j(\O)}\varphi_j(x)\varphi_j(y).$$
We note that the eigenfunctions are in $L^p(\O)$ for all $1\le p\le\infty$. It follows by Parseval's formula that
\be\label{a10}
Q_\O(t)=\sum_{j=1}^{\infty}e^{-t\lambda_j(\O)}
\left(\int_\O\varphi_j\,dx\right)^2\le
e^{-t\lambda_1(\O)}\sum_{j=1}^{\infty}
\left(\int_\O\varphi_j\,dx\right)^2=e^{-t\lambda_1(\O)}|\O|.
\ee
Since the torsion function is given by
$$w_\O(x)=\int_0^\infty u_\O(x;t)\,dt,$$
we have that
$$T(\O)=\sum_{j=1}^{\infty}\lambda_j(\O)^{-1}
\left(\int_\O\varphi_j\,dx\right)^2.$$
We recover Proposition 2.3. by integrating \eqref{a10} with respect to $t$ over $[0,\infty)$:
$$T(\O)\le\lambda_1(\O)^{-1}\sum_{j=1}^{\infty}
\left(\int_\O\varphi_j\,dx\right)^2=\lambda_1(\O)^{-1}|\O|.$$

Let $M_1$ and $M_2$ be two open sets in Euclidean space with finite Lebesgue measures $|M_1|$ and $|M_2|$ respectively. Let $M=M_1\times M_2$. We have that
$$p_{M_1\times M_2}(x,y;t)=p_{M_1}(x_1,y_1;t)p_{M_2}(x_2,y_2;t),$$
where $x=(x_1,x_2), y=(y_1,y_2)$. It follows that
\be\label{a15}
Q_M(t)=Q_{M_1}(t)Q_{M_2}(t),
\ee
and
\be\label{a151}
T(M)=\int_0^{\infty}Q_{M_1}(t)Q_{M_2}(t)\,dt.
\ee
Integrating \eqref{a15} with respect to $t$, and using \eqref{a8} for $M_2$ we obtain that
\be\label{a16}
T(M)\le T(M_1)|M_2|.
\ee

This upper bound should be ``sharp'' if the decay of $Q_{M_2}(t)$ with respect to $t$ is much slower than the decay of $Q_{M_1}(t)$. The result below makes this assertion precise in the case where $M_2$ is a convex set with $\h(\partial M_2)<\infty$. The latter condition is for convex sets equivalent to requiring that $M_2$ is bounded. Here $\h$ denotes the $(d_2-1)$-dimensional Hausdorff measure.

\begin{teo}\label{the2}
Let $M=M_1\times M_2$, where $M_1$ is an arbitrary open set in $\R^{d_1}$ with finite $d_1$-measure and $M_2$ is a bounded convex open set in $\R^{d_2}$. Then there exists a constant $\C_{d_2}$ depending on $d_2$ only such that
\be\label{a17}
T(M)\ge T(M_1)|M_2|-\C_{d_2}\lambda_1(M_1)^{-3/2}|M_1|\h(\partial M_2).
\ee
\end{teo}

For the proof of Theorem \ref{the2} we need the following lemma (proved as Lemma 6.3 in \cite{vdB8}).

\begin{lemma}\label{lem3}
For any open set $\O$ in $\R^d$,
\be\label{a19}
u_\O(x;t)\ge1-2\int_{\{y\in
\R^d:|y-x|>d(x)\}}p_{\R^d}(x,y;t)\,dy,
\ee
where
$$d(x)=\min\{|x-z|\ :\ z\in\partial\O\}.$$
\end{lemma}

\begin{proof}[Proof of Theorem \ref{the2}]
With the notation above we have that
\begin{align*}
T(M)&=T(M_1)|M_2|-\int_0^\infty Q_{M_1}(t)(|M_2|-Q_{M_2}(t))\,dt\nonumber\\
&=T(M_1)|M_2|-\int_0^\infty Q_{M_1}(t)\int_{M_2}(1-u_{M_2}(x_2;t))\,dx_2\,dt.
\end{align*}
Define for $r>0$,
$$\partial M_2(r)=\{x\in M_2:d(x)=r\}.$$
It is well known that (Proposition 2.4.3 in \cite{bubu05}) if $M_2$ is convex then
\be\label{a22}
\h(\partial M_2(r))\le\h(\partial M_2).
\ee
By \eqref{a10}, \eqref{a19} and \eqref{a22} we obtain that
\begin{align}\label{a23}
\int_0^{\infty}&Q_{M_1}(t)\int_{M_2}(1-u_{M_2}(x_2;t))\,dx_2\,dt\nonumber\\
&\le2|M_1|\h(\partial M_2)\int_0^{\infty}dt\,
e^{-t\lambda_1(M_1)}\int_0^{\infty}dr\int_{\{z\in
\R^{d_2}:|z-x|>r\}}p_{\R^{d_2}}(x,z;t)\,dz\nonumber\\
&=2d_2\omega_{d_2}|M_1|\h(\partial M_2)\int_0^{\infty}dt\,
e^{-t\lambda_1(M_1)}(4\pi t)^{-d_2/2}\int_0^{\infty}dr\,r^{d_2}e^{-r^2/(4t)}\nonumber\\
&=\C_{d_2}\lambda_1(M_1)^{-3/2}|M_1|\h(\partial M_2),
\end{align}
where
$$\C_{d_2}=\frac{\pi^{1/2}d_2\Gamma((d_2+1)/2)}{\Gamma((d_2+2)/2)}.$$
This concludes the proof.
\end{proof}

\begin{proof}[Proof of Proposition \ref{prop}]
Let $M_1=(0,\epsilon)\subset\R$, $M_2=\omega\subset\R^{d-1}$. Since the
torsion function for $M_1$ is given by $x(\epsilon-x)/2,\ 0\le
x\le \epsilon$ we have that $T(M_1)=\epsilon^3/{12}$. Then
\eqref{a16} proves the upper bound. The lower bound follows from
\eqref{a17} since $\lambda_1(M_1)=\pi^2/{\epsilon}^2$, $|M_1|=\epsilon$.
\end{proof}

It is of course possible, using the Faber-Krahn inequality for $\lambda_1(M_1)$, to obtain a bound for the right-hand side of \eqref{a23} in terms of $|M_1|^{(d_1+3)/{d_1}}\h(\partial M_2)$.

Our next result is an improvement of Proposition \ref{rectangle}. The torsional rigidity for a rectangle follows by substituting the formulae for $Q_{(0,a)}(t)$ and $Q_{(0,b)}(t)$ given in
\eqref{a27} below into \eqref{a151}. We recover the expression given on p.108 in \cite{PS}:
$$T(R_{a,b})=\frac{64ab}{\pi^6}\sum_{k=1,3,\cdots}\sum_{l=1,3,\cdots}
k^{-2}l^{-2}\left(\frac{k^2}{a^2}+\frac{l^2}{b^2}\right)^{-1}.$$
Nevertheless the following result is not immediately obvious.

\begin{teo}\label{the3}
\be\label{a25}
\left|T(R_{a,b})-\frac{a^3b}{12}+\frac{31\zeta(5)a^4}{2\pi^5}\right|
\le\frac{a^5}{15b},
\ee
where
$$\zeta(5)=\sum_{k=1}^{\infty}\frac{1}{k^5}.$$
\end{teo}

\begin{proof}
A straightforward computation using the eigenvalues and eigenfunctions of the Dirichlet Laplacian on the interval together with the first identity in \eqref{a10} shows that
\be\label{a27}
Q_{(0,a)}(t)=\frac{8a}{\pi^2}\sum_{k=1,3,\dots}k^{-2}e^{-t\pi^2k^2/a^2}.
\ee
We write
\be\label{a28}
Q_{(0,b)}(t)=b-\frac{4t^{1/2}}{\pi^{1/2}}+\left(Q_{(0,b)}(t)+\frac{4t^{1/2}}{\pi^{1/2}}-b\right).
\ee
The constant term $b$ in the right-hand side of \eqref{a28} gives, using \eqref{a27}, a contribution
\begin{align*}
\frac{8ab}{\pi^2}&\int_{[0,\infty)}dt\sum_{k=1,3,\dots}k^{-2}e^{-t\pi^2k^2/a^2}=\frac{8a^3b}{\pi^4}\sum_{k=1,3,\dots}k^{-4}\nonumber
\\&
=\frac{8a^3b}{\pi^4}\left(\sum_{k=1}^{\infty}k^{-4}-\sum_{k=2,4,\dots}k^{-4}\right)
=\frac{15a^3b}{2\pi^4}\zeta(4)\nonumber \\ &= \frac{a^3b}{12},
\end{align*}
which jibes with the corresponding term in \eqref{a25}. In a very similar calculation we have that the $-\frac{4t^{1/2}}{\pi^{1/2}}$ term in the right-hand side of \eqref{a28} contributes
\begin{align*}
-\frac{32a}{\pi^{5/2}}\int_{[0,\infty)}dt\,t^{1/2}\sum_{k=1,3,\dots}k^{-2}e^{-t\pi^2k^2/a^2}=-\frac{31\zeta(5)a^4}{2\pi^5},
\end{align*}
which jibes with the corresponding term in \eqref{a25}. It remains to bound the contribution from the expression in the large round brackets in \eqref{a25}. Applying formula \eqref{a27} to the interval $(0,b)$ instead and using the fact that $\sum_{k=1,3,\cdots}k^{-2}=\pi^2/8$ gives that
\begin{align}\label{a31}
Q_{(0,b)}(t)-b+\frac{4t^{1/2}}{\pi^{1/2}}&=\frac{8b}{\pi^2}\sum_{k=1,3,\dots}k^{-2}\left(e^{-t\pi^2k^2/b^2}-1\right)+\frac{4t^{1/2}}{\pi^{1/2}}\nonumber\\&
=-\frac{8}{b}\sum_{k=1,3,\dots}\int_{[0,t]}d\tau
e^{-\tau\pi^2k^2/b^2}+\frac{4t^{1/2}}{\pi^{1/2}}\nonumber \\ &
=-\frac{8}{b}\int_{[0,t]}d\tau\left(\sum_{k=1}^{\infty}e^{-\tau\pi^2k^2/b^2}-\sum_{k=1}^{\infty}e^{-4\tau\pi^2k^2/b^2}\right)+\frac{4t^{1/2}}{\pi^{1/2}}.
\end{align}
In order to bound the right-hand side of \eqref{a31} we use the following instance of the Poisson summation formula.
$$\sum_{k\in \mathbb{Z}}e^{-t\pi k^2}=t^{-1/2}\sum_{k\in
\mathbb{Z}}e^{-\pi k^2/t}, \ t>0.$$
We obtain that
$$\sum_{k=1}^{\infty}e^{-t\pi
k^2}=\frac{1}{(4t)^{1/2}}-\frac{1}{2}+t^{-1/2}\sum_{k=1}^{\infty}e^{-\pi
k^2/t},\ t>0.$$
Applying this identity twice (with $t=\pi \tau/b^2$ and $t=4\pi\tau/b^2$ respectively) gives that the right-hand side of \eqref{a31} equals
\begin{align*}
-\frac{8}{\pi^{1/2}}\int_{[0,t]}d\tau
\left(\tau^{-1/2}\sum_{k=1}^{\infty}e^{-k^2b^2/{\tau}}-(4\tau)^{-1/2}\sum_{k=1}^{\infty}e^{-k^2b^2/{(4\tau)}}\right).
\end{align*}
Since $k\mapsto e^{-k^2b^2/{\tau}}$ is non-negative and decreasing,
$$\sum_{k=1}^{\infty}\tau^{-1/2}e^{-k^2b^2/{\tau}}\le\tau^{-1/2}\int_{[0,\infty)}dke^{-k^2b^2/{\tau}}=\pi^{1/2}(2b)^{-1}.$$
It follows that
$$\left|Q_{(0,b)}(t)-b+\frac{4t^{1/2}}{\pi^{1/2}}\right|\le\frac{8t}{b},\ t>0.$$
So the contribution of the third term in \eqref{a28} to
$T(R_{a,b})$ is bounded in absolute value by
\begin{align*}
\frac{64a}{\pi^2b}\int_{[0,\infty)}dt\,t\sum_{k=1,3,\dots}k^{-2}e^{-t\pi^2k^2/a^2}
&=\frac{64a^5}{\pi^6b}\sum_{k=1,3,\dots}k^{-6}\nonumber\\
&=\frac{63a^5}{\pi^6b}\zeta(6)\nonumber\\
&=\frac{a^5}{15b}.
\end{align*}
This completes the proof of Theorem \ref{the3}.
\end{proof}

The Kohler-Jobin theorem mentioned in Section \ref{sposit} generalizes to $d$-dimensions: for any open set $\O$ with finite measure the ball minimizes $T(\O)\lambda_1(\O)^{(d+2)/2}$. Moreover, in the spirit of Theorem \ref{the2}, the following inequality is proved in \cite{vdB7} through an elementary heat equation proof.

\begin{teo}\label{the4}
If $T(\O)<\infty$ then the spectrum of the Dirichlet Laplacian
acting in $L^2(\O)$ is discrete, and
$$T(\O)\ge \left(\frac{2}{d+2}\right)\left(\frac{4\pi
d}{d+2}\right)^{d/2}\sum_{k=1}^{\infty}\lambda_k(\O)^{-(d+2)/2}.$$
\end{teo}

We obtain, using the Ashbaugh-Benguria theorem (p.86 in \cite{hen06}) for $\lambda_1(\O)/{\lambda_2(\O)}$, that
\be\label{a38}
T(\O)\lambda_1(\O)^{(d+2)/2}\ge\left(\frac{2}{d+2}\right)\left(\frac{4\pi
d}{d+2}\right)^{d/2}\Gamma\left(1+\frac{d}{2}\right)\left(1+\left(\frac{\lambda_1(B)}{\lambda_2(B)}\right)^{(d+2)/2}\right).
\ee
The constant in the right-hand side of \eqref{a38} is for $d=2$ off by a factor $\frac{j_{0,1}^4j_{1,1}^4}{8(j_{0,1}^4+j_{1,1}^4)}\approx 3.62$ if compared with the sharp Kohler-Jobin constant. We also note the missing factor $m^{m/(m+2)}$ in the right-hand side of (57) in \cite{vdB7}.

\bigskip
\ack A large part of this paper was written during a visit of the first two authors at the Isaac Newton Institute for Mathematical Sciences of Cambridge (UK). GB and MvdB gratefully acknowledge the Institute for the excellent working atmosphere provided. The authors also wish to thank Pedro Antunes helpful discussions. The work of GB is part of the project 2010A2TFX2 {\it``Calcolo delle Variazioni''} funded by the Italian Ministry of Research and University.


\bigskip
{\small\noindent
Michiel van den Berg:
School of Mathematics,
University of Bristol\\
University Walk,
Bristol BS8 1TW - UK\\
{\tt mamvdb@bristol.ac.uk}\\
{\tt http://www.maths.bris.ac.uk/~mamvdb/}

\bigskip\noindent
Giuseppe Buttazzo:
Dipartimento di Matematica,
Universit\`a di Pisa\\
Largo B. Pontecorvo 5,
56127 Pisa - ITALY\\
{\tt buttazzo@dm.unipi.it}\\
{\tt http://www.dm.unipi.it/pages/buttazzo/}

\bigskip\noindent
Bozhidar Velichkov:
Dipartimento di Matematica,
Universit\`a di Pisa\\
Largo B. Pontecorvo 5,
56127 Pisa - ITALY\\
{\tt b.velichkov@sns.it}
\end{document}